\documentclass[letter,12pt]{article}

\usepackage{amssymb}
\usepackage{amsthm}
\usepackage{amsfonts}
\usepackage{amsmath}
\usepackage{hyperref}
\usepackage{bbm}
\usepackage[top=0.9in, bottom=0.9in, left=0.9in, right=0.9in]{geometry}

\newtheorem*{theorem*}{Theorem}

\newtheorem{theorem}{Theorem}[section]

\newtheorem{lemma}[theorem]{Lemma}
\newtheorem{corollary}[theorem]{Corollary}

\newtheorem{remark}[theorem]{Remark}

\newtheorem{proposition}[theorem]{Proposition}

\def\del{\partial}
\def\dbar{\bar\partial}
\def\ddbar{\del\dbar}

\def\del{\partial}

\def\o{\omega}

\def\beq{\begin{equation}}
\def\eeq{\end{equation}}

\title{The Mabuchi geometry of low energy classes}
\author{Tam\'as Darvas}
\date{\emph{\small{to the memory of Gabriela Kohr (1967-2020)}}}
\begin{document}
\maketitle
\begin{abstract}Let $(X,\omega)$ be a K\"ahler manifold and $\psi: \Bbb R \to \Bbb R_+$ be a concave weight. We show that $\mathcal H_\omega$ admits a natural metric $d_\psi$ whose completion is the low energy space $\mathcal E_\psi$, introduced by Guedj-Zeriahi. As $d_\psi$ is not induced by a Finsler metric, the main difficulty is to show that the triangle inequality holds. We study properties of the resulting complete metric space $(\mathcal E_\psi,d_\psi)$.
\end{abstract}

\section{Introduction}

Let $(X,\omega)$ be a compact connected K\"ahler manifold. A basic problem in K\"ahler geometry is to find various canonical metrics among the K\"ahler metrics $\omega'$ that are in the same de Rham cohomology class as $\omega$ \cite{Ca,Ya}. Due to Hodge theory, such $\omega'$ can be written as $\omega' = \omega + i\ddbar u$, where $u$ is a smooth function from the space of K\"ahler potentials:
$$\mathcal H_\omega := \{v \in C^\infty(X) \ : \ \omega_v:=\omega + i\ddbar v > 0\}.$$
When studying weak notions of K\"ahler metrics, or degenerations of smooth ones, a natural space to consider is the space of $\omega$-plurisubharmonic ($\omega$-psh) functions $\textup{PSH}(X,\omega)$. With slight abuse of precision, we say that $v: X \to [-\infty,\infty)$ is $\omega$-psh if it is usc, integrable and 
$\omega_v:=\omega + i\ddbar v \geq 0$ in the sense of currents.

As pointed out in a series of works by Guedj-Zeriahi, and their collaborators \cite{GZ07,BEGZ10,BBGZ13,EGZ09} the space of full mass potentials $\mathcal E = \{v \in \textup{PSH}(X,\omega) \ :  \ \int_X \omega_v^n = \int_X \omega^n\}$ has a prominent role in the study of weak solutions to complex Monge-Amp\`ere equations with measure theoretic right hand side. Here $
\omega_v^n$ is the non-pluripolar complex Monge-Ampere measure of $v$, extending the interpetation of Bedford-Taylor and Cegrell from the local case \cite{BT76,Ce98}. To study $\mathcal E$, it is feasible to consider various weights $\phi: \Bbb R \to \Bbb R^+$ and consider the subspaces with finite $\phi$-energy:
$$\mathcal E_\phi := \{v \in \mathcal E \ | \ E_\phi(v):=\int_X \phi(v)\omega_v^n < \infty\}.$$
One important point is that $\mathcal E$ can be exhausted by a special class of finite energy subspaces \cite[Proposition 2.2]{GZ07}:
\begin{equation}\label{eq: E_union}
\mathcal E = \bigcup_{\psi \in \mathcal W^-} \mathcal E_\psi.
\end{equation}
Here $\mathcal W^-$ is the space of weights $\phi : [-\infty,\infty] \to [0,\infty]$ that are even and continuous on $\Bbb R$, in addition to being smooth, concave, and strictly increasing on $(0,\infty)$, normalized by $\psi(0) =0$ and $\psi(\pm \infty) = \infty$. Following terminology of \cite{GZ07}, we will call elements of $\mathcal W^-$ \emph{concave weights}, and the spaces $\mathcal E_\psi$ \emph{low energy classes} (see Remark \ref{rem: no smooth} for superficial differences in our discussion compared to \cite{GZ07}).

As noticed in \cite[Section 2]{BBEGZ11}, the subspace $\mathcal E_1 \:= \{ v \in \mathcal E \ : \ \int_X |v| \omega_v^n <\infty\}$ has a complete metric topology. This was refined further in \cite{Da15}, where it was noticed that the more general \emph{high energy classes} $\mathcal E_\chi$ are the metric completions of an appropriately defined Orlicz-Finsler metric structure on the smooth space $\mathcal H_\omega$. Recall that high energy classes $\mathcal E_\chi$ are given by weights $\chi: \Bbb R \to \Bbb R^+$, that are even convex functions satisfying $\chi(0)=0, \chi'(1)=1$ and a growth estiamte $t\chi'(t) \leq p \chi(t)$ (notation: $\chi \in \mathcal W^+_p$).

Moreover, in \cite{Da15} it was also pointed out the the resulting metric spaces $(\mathcal E_\chi,d_\chi)$ admit geodesic segments connecting arbitrary points. This latter fact had a wide range of applications, including energy properness \cite{CC1,CC2,BDL20,DR17}, K-stability \cite{BBJ15}, convergence and existence the weak Calabi flow \cite{Str, BDL17}, etc.

Unfortunately, for all $\chi \in \mathcal W^+_p$ we have the inclusions $\mathcal E_\chi \subseteq \mathcal E_1 \subsetneq \mathcal E$. As a result, it is natural to ask if subspaces of $\mathcal E$ not included in $\mathcal E_1$ can be naturally topologized/geometrized as well. One may even ask: does $\mathcal E$ admit a natural topology/geometry? Revisiting \eqref{eq: E_union}, one is tempted to first find a natural metric topology on the low energy spaces $\mathcal E_\psi$, as these exhaust $\mathcal E$. This is what we accomplish in this paper.

Not much is known about the metric geometry of low energy spaces, despite their wast array of applications to weak solutions of complex Monge-Amp\`ere equations \cite{BEGZ10,GZ07}. The only related result seems to be \cite[Theorem 1.6]{GLZ17}, implying existence of a metrizable uniform space topology on $\mathcal E_p := \{u \in \mathcal E, \ \int_X |u|^p \omega_u^n<\infty\}, \ p \in(0,1)$.

To start, let $\mathcal H_\omega^\Delta := \textup{PSH}(X,\omega) \cap C^{1,\bar 1}$, where by $C^{1,\bar 1}$ we denote functions on $X$ with bounded mixed second partial derivatives. Equivalently, $\mathcal H_\omega^\Delta$ is the space of $\omega$-psh potentials with bounded Laplacian. Given $u_0,u_1 \in \mathcal H_\omega$, let $[0,1] \ni t \to u_t \in \mathcal H_\omega^\Delta$ be Chen's weak geodesic joining $u_0,u_1$ \cite{Ch00}. We introduce the following candidate metric on $\mathcal H_\omega$:
\begin{equation}\label{eq: d_psi_def_main}
d_\psi(u_0,u_1) := \int_X \psi(\dot u_0) \omega_{u_0}^n.
\end{equation}
The above definition of $d_\psi$ bears superficial similarities with the one in \cite[(2)]{Da15},  dealing with the case of high energy classes. However it is not difficult to see that $d_\psi$ is not induced by the length metric of a Finsler structure, contrasting with \cite{Da15}. Thus one has to work hard to prove the triangle inequality, this being the first main result of this paper:

\begin{theorem}\label{thm: d_psi_metric_main} $(\mathcal H_\omega,d_\psi)$ is a metric space.
\end{theorem}

Hoping for further analogies with the case of high energy classes \cite{Da15}, one might mistakenly expect that $(\mathcal H_\omega,d_\psi)$ is at least a length space, and the weak geodesic $t \to u_t$ appearing in \eqref{eq: d_psi_def_main} is a metric $d_\psi$-geodesic connecting $u_0,u_1$. This is unfortunately not the case either. In fact, when $\psi(t) = |t|^{\alpha}, \alpha \in (0,1)$, one can easily verify that the $d_\psi$-length of smooth curves inside $\mathcal H_\omega$ is always zero. In addition, this also confirms that $d_\psi$ can not be induced by a Finsler metric on $\mathcal H_\omega$.

In fact, the right analogy to follow here is the one coming from the case of toric K\"ahler manifolds $(X_T,\omega_T)$, and restricting $d_\psi$ to the torus invariant potentials $\mathcal H_\omega^T$. As is well known, the Legendre transform $\mathcal L$  transforms $\mathcal H_\omega^T$ bijectively into the space $Conv_\omega(P)$ of smooth convex functions on a Delzant polytope $P \subset \Bbb R^n$ with specific asymptotics near $\partial P$ (for details see \cite[Section 4]{Gu14} or \cite{CGSZ19}). By the same calculations as in \cite[Proposition 4.3]{Gu14}, we obtain that for $u_0,u_1 \in \mathcal H_\omega^T$ we have
$$d_\psi(u_0,u_1) = \int_P \psi(\mathcal L(u_0) - \mathcal L(u_1)) d\mu,$$
with $\mu$ being the Lebesque measure on $P$. By Lemma \ref{lem: sublinear} below we immediately see that in this case $d_\psi$ satisfies the triangle inequality trivially. 

In addition, the $d_\psi$-completion of $\mathcal H_\omega^T$ will be $L^\psi(P) \cap Conv(P)$, the space of convex functions on $P$ that have finite $\psi$-integral. This space is exactly the Legendre dual of $\mathcal E_\psi^T$, the set of torus invariant potentials in $\mathcal E_\psi$ (\cite[Proposition 4.5]{Gu14}). 

With the toric analogies in mind, our reader is perhaps less surprised by the statement of Theorem \ref{thm: d_psi_metric_main} above, and might also expect that the metric completion of $(\mathcal H_\omega,d_\psi)$ equals $\mathcal E_\psi$, even in the absence of toric symmetries. This is confirmed in our next main result.

\begin{theorem} \label{thm: EpsiComplete_main}  The metric $d_\psi$ extends to $\mathcal E_\psi$, making $(\mathcal E_\psi,d_\psi)$ a complete metric space, that is the metric completion of $(\mathcal H_\o,d_\psi)$. 
\end{theorem}

This result is analogous to \cite[Theorem 2]{Da15} that deals with the case of high energy classes. The similarities don't stop here. Paralleling \cite[Theorem 3]{Da15}, the $d_\psi$ metric is comparable to a concrete analytic expression:

\begin{theorem} \label{thm: Energy_Metric_Eqv_main} For any $u_0,u_1 \in \mathcal E_\psi$ we have
\begin{equation}\label{eq: Energy_Metric_Eqv_main}
d_\psi(u_0,u_1) \leq \int_X \psi(u_0 - u_1) \o_{u_0}^n + \int_X \psi(u_0 - u_1) \o_{u_1}^n\leq {2^{2n + 5}} d_\psi(u_0,u_1).
\end{equation}
\end{theorem}

This result implies that the expression $I_\psi(u_0,u):=\int_X \psi(u_0 - u_1) \o_{u_0}^n + \int_X \psi(u_0 - u_1) \o_{u_1}^n$ satisfies a quasi-triangle inequality, a result of independent interest. Previously this was obtained using analytic methods for the weights $\psi(t) = |t|^p, \ p \in (0,1)$ in \cite[Theorem 1.6]{GLZ17}.

As pointed about above, the absence of a background Finsler structure requires a new approach to the proof of Theorem \ref{thm: d_psi_metric_main}. However once the triangle inequality is obtained, many pluripotential theoretic arguments can be used from \cite{Da15}, and this will be apparent in the proofs of Theorems \ref{thm: EpsiComplete_main} and \ref{thm: Energy_Metric_Eqv_main}.

Contrasting with the case of high energy classes explored in \cite{Da15}, our methods suggest that the metric $d_\psi$ is somehow positively curved (see Proposition \ref{prop: concavity} and Corollary \ref{cor: concavity}). However it remains to be seen if such a notion can be defined for non-geodesic metric spaces, as it is the case here. Since weak geodesic segments are used to define $d_\psi$ in \eqref{eq: d_psi_def_main}, it could be beneficial to understand what role these curves play from a metric/geometric point of view. In a different direction, it would be interesting to extend our results to more singular spaces. There has been a flurry of activity in this latter area recently, focusing mostly on the high energy case \cite{DNL20,GT22, Xi19,Tr19, Tr20}. Lastly, we are curious if the quantization scheme of the high energy spaces from \cite{DLR20} has an analogue in our low energy context. We hope to investigate these questions, as well as applications in future works.

\paragraph{Organization.} In Section 2 we recall known results about finite energy classes, and obtain the second order variation of low energy weak quasi-norms. After some preliminary results on our candidate metric $d_\psi$ in Section 3, we prove the triangle inequality (and Theorem 1.1) in Section 4. Theorem 1.2 is proved in Section 5. Finally, Theorem 1.3 is proved in Section 6.

\paragraph{Acknowledgements.}

I dedicate this work to the memory of Gabriela Kohr, my master's thesis advisor. Her enthusiasm towards scientific discovery will be dearly missed, as well as her passion for educating the next generation of analysts at Babe\c s-Bolyai University. I mourn her sudden loss that is especially devastating to our community.

We thank L\'aszl\'o Lempert for extensive conversations related to the topic of this paper and his earlier work \cite{Le20}, that provided the motivating spark for this paper. We also thank him for allowing us to include his argument for Proposition \ref{prop: maxilengthgeod}, that is much simpler than what we originally had in mind. We thank Prakhar Gupta for carefully reading a preliminary version of the manuscript, and pointing out numerous imprecisions. We also thank the anonymous referee for the suggestions that improved our paper.

Research partially supported by a Sloan Fellowship and National Science Foundation grant DMS-1846942.

\section{Preliminaries}

Most of our notation and terminology builds on that of \cite{GZ07}, \cite{Da15} and the survey \cite{Da19}. We refer our reader to these works for a detailed background.  Below we only recall the basics, adapted to our specific context.

\subsection{Finite energy classes} 

We recall here some basic facts about the class $\mathcal E \subset \text{PSH}(X,\o)$ and its subspaces.  We refer to the original papers \cite{GZ07} and \cite{BEGZ10} and the recent book \cite{gzbook} for a complete picture. For $v \in \text{PSH}(X,\o)$, the  canonical cutoffs $v_h, \ h \in \Bbb R$ are given by the formula $v_h := \max (-h, v ).$ By an application of the comparison principle of Bedford-Taylor theory, it follows that the Borel measures $\mathbbm{1}_{\{v >-h\}}(\o + i\partial \bar\partial v_h)^n$ are increasing in $h$. As a result, one can make make sense of $(\o + i\partial\bar\partial v)^n$ as the limit of these increasing measures, even if $v$ is unbounded:
\begin{equation}\label{non-plurip}\o_v^n := (\o + i\partial\bar\partial v)^n= \lim_{h \to \infty} \mathbbm{1}_{\{v >-h\}}(\o + i\partial \bar\partial v_h)^n.
\end{equation}
With this definition, $\o_v^n$ is called the non-pluripolar Monge-Amp\`ere measure of $v$. It follows from \eqref{non-plurip}) that
$$\int_X \o_v^n \leq \int_{X}\o^n =:V,$$
bringing us to the class of full mass potentials $\mathcal E$. By definition, $v \in \mathcal E$ if
\begin{equation}\label{Eps_def}
\int_X \o_v^n=\lim_{h \to \infty} \int_X \mathbbm{1}_{\{v >-h\}}(\o + i\partial \bar\partial v_h)^n =V.
\end{equation}

Suppose $\phi : [-\infty,\infty] \to [0,\infty]$ is a continuous even function, with $\phi(0)= 0$ and $\phi(\pm \infty)=\infty$. Such $\phi$ is referred to as a \emph{weight}. The set of all weights is denoted by $\mathcal W$. By definition, for $v \in \mathcal E$ we have $v \in \mathcal E_\phi$ if
$$E_\phi(v):=\int_X \phi(v) \o_v^n < \infty.$$
The two special classes of weights that are most interesting in the theory are:
\begin{flalign*}
\mathcal W^-&=\big\{\psi \in \mathcal W  \ \big| \ \psi \textup{ is concave, strictly increasing, and smooth on $(0,\infty)$}\},\\
\mathcal W^+_p&=\big\{\chi \in \mathcal W \ \big| \ \chi \textup{ is convex and } t\chi'(t) \leq p\chi(t), \ t \in \Bbb R\big\},\end{flalign*}
where $p \geq 1$. We note the sign difference between our convention for $\mathcal W^-$, $\mathcal W^+_p$, and the ones in \cite{GZ07} and \cite[Section 2.3]{Da17b}. 

Of particular importance are the weights $\chi_p(t)=|t|^p, p >0$, and the associated classes $\mathcal E_p:=\mathcal E_{\chi_p}$.
Note that $\chi_p \in \mathcal W^+_p$ for $p \geq 1$ and $\chi_p \in \mathcal W^-$ for $0 < p \leq 1$. The case $p=1$ interpolates between convex and concave energy classes since 
$$\mathcal E_\chi \subset \mathcal E_1 \subset \mathcal E_\psi,$$
for any $\chi \in \mathcal W^+_p$ and $\psi \in \mathcal W^-$.

In this work we will be focusing on the concave weights $\mathcal W^-$. As mentioned in the introduction, the interest in them comes from the following fact \cite[Proposition 2.2]{GZ07}:
\begin{equation}\label{eq: E_union2}
\mathcal E=\{ v \in \mathcal E_\psi \ | \ \psi \in \mathcal W^-\}.
\end{equation}

\begin{remark}\label{rem: no smooth} To be precise, in \cite{GZ07} the authors proved \eqref{eq: E_union2} for concave weights $\psi$ that are not necessarily smooth on $(0,\infty)$. However it is elementary to see that for  a non-smooth concave weight $\psi$ we can find another concave weight $\tilde \psi$, smooth on $(0,\infty)$, such that $\mathcal E_\psi = \mathcal E_{\tilde \psi}$. Indeed, one can even make sure that $\psi - \tilde \psi$ is bounded. Because of this, very little is gained from working with more general concave weights. For sake of brevity we leave it to the interested reader to work out the details of our results in the case when the weights $\psi \in \mathcal W^-$ are not assumed to be smooth on $(0,\infty)$. This can be carried out using approximation, much in the same way as it is done in \cite{Da15}. 
\end{remark}

The following result is sometimes called the ``fundamental estimate":

\begin{proposition} \textup{\cite[Lemma 2.3, Lemma 3.5]{GZ07}}\label{prop: Energy_est} Let $\phi\in \mathcal W^- \cup \mathcal W^+_p, \ p \geq1$. If $u,v \in \mathcal E_\phi$ with $u\leq v \leq 0$ then $$E_\phi(v) \leq  C E_\phi(u).$$ 
Here $C>0$ depends only on $p$.
\end{proposition}
If $\phi\in \mathcal W^- \cup \mathcal W^+_p, \ p \geq1$ then the $\phi$-energy has a very useful continuity property:

\begin{proposition}\textup{\cite[Proposition 5.6]{GZ07}} \label{E_semicont} Let $\phi\in \mathcal W^- \cup \mathcal W^+_p$ and $\{u_j\} _{j \in \Bbb N} \subset \textup{PSH}(X,\o) \cap L^\infty$ is a sequence decreasing to $u \in \textup{PSH}(X,\o)$. If
 $\sup_j E_\phi(u_j) <\infty$ then $u \in \mathcal E_\phi$. Moreover we have
$$E_\phi(u) = \lim_{j\to \infty}E_\phi(u_j).$$
\end{proposition}
Using the canonical cutoffs, the last two results imply  the very important ``monotonicity property":
\begin{corollary}\label{cor: Emonoton} Let $\phi\in \mathcal W^- \cup \mathcal W^+_p, \ p \geq1$. If $u \leq v$ and $u \in \mathcal E_\phi$ then $v \in \mathcal E_\phi$.
\end{corollary}

We note that the continuity property of the Monge-Amp\`ere operator from Bedford-Taylor theory \cite{BT76} is also preserved in this more general setting:
\begin{proposition}\textup{\cite[Theorem 2.17]{BEGZ10} \label{MA_cont}} Suppose  $\{ v_k\}_{k \in \Bbb N} \subset \mathcal E(X,\o)$ decreases (increases  a.e.) to $v \in \mathcal E(X,\o)$. Then $\o_{v_k}^n \to \o_v^n$ weakly.
\end{proposition}

A more general weak convergence result is proved in \cite[Proposition 2.20]{Da19}, and the remark following it.

\subsection{The $L^2$ metric and weak geodesics}

As introduced by Mabuchi, and independently by Semmes and Donaldson, $\mathcal H_\omega$ can be endowed with a natural infinite dimensional $L^2$-type Riemannian metric:
\begin{equation}\label{eq: MabL2}
\langle\alpha,\beta\rangle_u = \frac{1}{\int_X \omega^n}\int_X \alpha \beta \omega_u^n, \ \ \alpha,\beta \in T_u \mathcal H_\omega = C^\infty(X).
\end{equation}
One can compute the Levi-Civita connection $\nabla_{(\cdot)}(\cdot)$ of this inner-product and the associated geodesic equation.  For a thorough discussion of the $L^2$ Mabuchi--Semmes--Donaldson geometry, as well as its Levi-Civita connection, we refer to the surveys \cite[Section 4]{Bl13}, \cite[Section 3.1]{Da19}, as well as the original papers \cite{Ma87,Se92,Do99,Ch00}.

Unfortunately smooth geodesics connecting arbitrary $u_0,u_1 \in \mathcal H_\omega$ don't exist, but a weak notion of geodesic was studied by Chen \cite{Ch00}. His construction can be generalized to construct weak geodesic segments connecting points of $\text{PSH}(X,\o)\cap L^\infty(X)$. Following Berndtssson, we recall how this argument works.

As before, let $S \subset \Bbb C$ be the strip $\{ 0 < \textup{Re }s < 1\}$ and $\tilde \o$ be the pullback of $\o$ to the product $S \times X$. As argued in \cite[Section 2.1]{Bern13a}, for $u_0,u_1 \in \text{PSH}(X,\o)\cap L^\infty(X)$ the following Dirichlet problem has a unique solution:
\begin{alignat}{2}\label{bvp_Bern}
&u \in \text{PSH}(S\times X, \tilde\o)\cap L^\infty(S\times X)\nonumber\\
&(\tilde \o + i \partial \overline{\partial}u)^{n+1}=0 \nonumber\\
&u(t+ir,x) =u(t,x) \ \forall x \in X, t \in (0,1), r \in \Bbb R \\
&\lim_{t \to 0,1}u(t,x)=u_{0,1}(x),  \forall x \in X\nonumber.
\end{alignat}
Since the solution $u$ invariant in the imaginary direction, we denote it by $[0,1] \ni t\to u_t \in \text{PSH}(X,\o)\cap L^\infty$ and call it the weak geodesic joining $u_0$ and $u_1$.

In case $u_0,u_1 \in \mathcal H_\omega$ in \cite{Ch00} it was proved that $\Delta u \in L^\infty(S \times X)$. Such a curve $[0,1] \ni t \to u_t \in \text{PSH}(X,\o)\cap C^{1,\bar 1}=: \mathcal H_\omega^{\Delta}$ is called a $C^{1,\bar 1}$-geodesic.

A curve $[0,1]\ni t \to v_t \in \text{PSH}(X,\o)$ is called a subgeodesic if $v(s,x):=v_{\textup{Re} s}(x) \in \text{PSH}(S\times X,\tilde \o)$. We recall that the solution $u$ of \eqref{bvp_Bern} is constructed as the upper envelope
\begin{equation}\label{udef1}
u = \sup_{v \in \mathcal S}v,
\end{equation}
where $\mathcal S$ is the following set of weak subgeodesics:
$$ \mathcal S = \{ (0,1) \ni t \to v_t \in \text{PSH}(X,\o) \textup{ is a subgeodesic with }\lim_{t \to 0,1}v_t \leq u_{0,1} \}.$$
For a thorough discussion of weak geodesics we refer to Section 3 in the survey \cite{Da19}.

\subsection{First and second order variation of weak quasi-norms}

To start, we observe that concave weights are subadditive:

\begin{lemma}\label{lem: sublinear}Let $\psi \in \mathcal W^-$. Then $\psi$ is subadditive, i.e., $\psi(a+b) \leq \psi(a) + \psi(b), \ a,b \in \Bbb R$.
\end{lemma}
\begin{proof} Since $\psi$ is increasing on $[0,\infty)$ we have that $\psi(a+b)= \psi(|a+b|) \leq \psi(|a| + |b|)$. On the other hand, since $\partial_+ \psi$ is decreasing on $(0,\infty)$, we can finish the proof in the following way:
$$\psi(|a|) = \psi(|a|) - \psi(0) =  \int_0^{|a|} \partial_+ \psi(t) dt \geq \int_{|b|}^{|a| + |b|} \partial_+ \psi(t) dt = \psi(|a| + |b|) - \psi(|b|).$$
\end{proof}

To any concave weight $\psi \in \mathcal W^-$, and a finite measure space $(Y, \mu)$, one can associate  the space $L^\psi_\mu$. These will be $\mu$-measurable functions $f$, such that $\int_Y \psi(f)d\mu < \infty$. For such functions $f$, we can associate the a \emph{weak quasi-norm} that is only homogeneous, and typically does not satisfy (even weaker versions of) the triangle inequality:

\begin{equation}\label{eq: quasi-norm-def}\| f\|_{\psi,\mu} := \inf \{N>0 \ \textup{ s.t. }  \int_Y \psi(f/N) d\mu \leq 1 \}.
\end{equation}

When $\psi \in \mathcal W^+_p$, the above quantity  does define a bona-fide norm, and these are used in the K\"ahler geometry literature for approximation of $L^p$ Finsler metrics \cite{Da15,DL20}. Despite the fact that in our case the triangle inequality fails, these weak quasi-norms will still be important in our discussion.  
To note, compared to \cite[(13)]{Da15}, our definition in \eqref{eq: quasi-norm-def} is slightly different. There, to obtain the H\"older inequality \cite[(14)]{Da15}, we needed a version of \eqref{eq: quasi-norm-def} that is invariant with respect to taking scalar multiples of $\psi$. Our definition here is intentionally not scale invariant, since we need exactly this property in the last step of the proof of Proposition \ref{prop: concavity} below.

Given $u \in \mathcal H_\omega$ and $f \in L^\psi_{\omega_u^n}$, we will denote $\| f\|_{\psi,\omega_u^n}$ simply as $\|f \|_{\psi,u}$. To start, we note the following elementary convergence result.

\begin{lemma} \label{lem: convergence}Let $\mu$ and $\mu_k$ be finite Borel measures on $Y$. Let $\psi \in \mathcal W^-$, and $f_k,f$ be bounded functions that are $\mu_k$-measurable and $\mu$-measurable respectively. If $\int_Y \psi(c f_k) d\mu_k \to \int_Y \psi(c f) d\mu$ for all $c \in [0,\infty)$, then $\| f_k\|_{\psi,\mu_k} \to \| f\|_{\psi,\mu}$.
\end{lemma}

\begin{proof} We can assume that $f \not\equiv 0$ (a.e. with respect to $\mu$).  In this case $[0,\infty) \ni c \to \int_Y \psi(c f) d\mu \in \Bbb R$ is strictly increasing and continuous (the latter by the dominated convergence theorem). As a result, for any $\varepsilon >0$ there exists $\delta^1_\varepsilon,\delta^2_\varepsilon >0$ such that:
$$1+ \frac{\varepsilon}{2} \leq\int_Y \psi\bigg(\frac{f}{\|f \|_{\psi,\mu} - \delta^1}\bigg)d\mu \leq 1 + \varepsilon \  \textup{ and } \ 1-\varepsilon \leq\int_Y \psi\bigg(\frac{f}{\|f \|_{\psi,\mu} + \delta^2}\bigg)d\mu \leq 1- \frac{\varepsilon}{2}.$$
In addition, $\delta^1_\varepsilon,\delta^2_\varepsilon \searrow 0$ as $\varepsilon \searrow 0$.

By our assumption we have that
$\int_Y \psi\big({f_k}/{(\|f \|_{\psi,\mu} - \delta^1)}\big)d\mu_k  \to \int_Y \psi({f}/{(\|f \|_{\psi,\mu} - \delta^1)})d\mu$ and $\int_Y \psi({f_k}/{(\|f \|_{\psi,\mu} + \delta^2)})d\mu_k \to \int_Y \psi({f}/{(\|f \|_{\psi,\mu} + \delta^2)})d\mu.$ By definition of our weak quasi-norm we conclude that 
$\| f\|_{\psi, \mu} - \delta^1_\varepsilon \leq \liminf_k \| f_k\|_{\psi,\mu_k} \leq \limsup_k \| f_k\|_{\psi,\mu_k} \leq \| f\|_{\psi, \mu} + \delta^2_\varepsilon,$ finishing the proof. Letting $\varepsilon \searrow 0$, the result follows.
\end{proof}

\begin{lemma} \label{lem: convergence2}Let $f_k,f$ be continuous functions on a compact topological space $Y$ and $f_k \to f$ uniformly. Let $\mu$ and $\mu_k$ be Borel measures on $Y$ with finite mass such that $\mu_k \to \mu$ weakly. Then $\int_Y \psi(c f_k) d\mu_k \to \int_Y \psi(c f) d\mu$ for any $c \in [0,\infty)$ and $\| f_k\|_{\psi,\mu_k} \to \| f\|_{\psi,\mu}$.
\end{lemma}

\begin{proof} For any $c>0$ we have that $\psi(cf_k) \to \psi(cf)$ uniformly. Since we are dealing with finite measure spaces and $Y$ is compact, it follows that $\int_Y \psi(c f_k) d\mu_k \to \int_Y \psi(c f) d\mu$. The last claim follows from Lemma \ref{lem: convergence}.
\end{proof}

In case we have smooth maps $[0,1] \ni t \to v_t \in \mathcal H_\omega$, $[0,1] \ni t \to f_t \in C^\infty$ with $f_t >0$, it is easy to see that $t \to \| f_t\|_{\psi,v_t}$ is smooth. Indeed, since $\psi|_{(0,\infty)}$ is smooth, the arguments of \cite[Proposition 3.1]{Da15} carry over without change, and we have the following precise formula for the first derivative:

\begin{proposition}\label{prop: OrliczNormDiff} Suppose $\psi \in \mathcal W^-$. Given a smooth curve $(0,1) \ni t \to u_t \in \mathcal H$, i.e. $u(t,x):=u_t(x) \in C^\infty((0,1)\times X)$, and a positive smooth vector field $(0,1) \ni t \to f_t \in C^\infty(X)$ along this curve, the following formula holds:
\begin{equation}\label{eq: OrliczNormDiffEq}
\partial_t \|f_t\|_{\psi,u_t} = \frac{\int_X \psi' \Big(\frac{f_t}{\|f_t\|_{\psi,u_t}}\Big)\nabla_{\dot u_t} f_t \omega_{u_t}^n}{\int_X \psi'\Big(\frac{f_t}{\|f_t\|_{\psi,u_t}}\Big) \frac{f_t}{\|f_t\|_{\psi,u_t}}  \omega_{u_t}^n},
\end{equation}
where $\nabla$ is the covariant derivative of the $L^2$ Mabuchi--Semmes--Donaldson metric \eqref{eq: MabL2}.
\end{proposition}

Recall that Chen's $\varepsilon$-geodesics are smooth curves $t \to u_t$ that solve the following elliptic equation \cite{Ch00}:
\begin{equation}\label{eq: epsgeodeq}
\nabla_{\partial_t u} \partial_t u \omega_u^n= \varepsilon \omega^n.
\end{equation} 
As pointed out in \cite{Ch00}, the advantage of $\varepsilon$-geodesics is that they are smooth, and approximate uniformly the weak $C^{1\bar 1}$-geodesic connecting $u_0,u_1 \in \mathcal H_\omega$ that solves \eqref{bvp_Bern}.

For this paper we need to compute the second order variation of the length of very special vector fields across $\varepsilon$-geodesics (c.f. \cite[Section 4]{Le20} and \cite[Section 5]{CC3}):

\begin{proposition}\label{prop: OrliczNormSecondDiff} Suppose $\psi \in \mathcal W^-$. Let $[0,1]^2 \ni (s,t) \to u(s,t) \in \mathcal H_\omega$ be smooth and an $\varepsilon$-geodesic in each $t$-direction, such that $\partial_s u >0$. The following formula holds:
\begin{flalign}\label{eq: OrliczNormSecondDiffEq}
&\partial^2_t  \|\partial_s u\|_{\psi,u} = \\
&=\frac{  \int_X \psi''(\eta)\Big( \|\partial_s u\|_{\psi,u} \big(\nabla_{\partial_t u}\eta \big)^2 + \frac{1}{ \|\partial_s u\|_{\psi,u}}\{\partial_s u,\partial_t u\}_{\omega_u}^2 + \frac{\varepsilon}{ \|\partial_s u\|_{\psi,u}} \langle \nabla_{\omega_u} \partial_s u , \nabla_{\omega_u} \partial_s u \rangle \frac{\omega^n}{\omega_{u}^n}  \Big) \omega_{u}^n}{\int_X \psi'(\eta) \eta \omega_{u}^n}, \nonumber
\end{flalign}
where $\{\cdot,\cdot\}_{\omega_u}^2$ is the Poisson bracket of $\omega_u$, and we introduced $\eta := \frac{\partial_s u}{\|\partial_s u\|_{\psi,u}}$, for simplicity. In particular, for fixed $s$, the map $t \to   \|\partial_s u\|_{\psi,u}$ is concave.
\end{proposition}

\begin{proof} The proof is a careful calculation of the derivative of the right hand side of \eqref{eq: OrliczNormDiffEq} in case $s \in [0,1]$ is fixed and $f_t := \partial_s u(s,t), \ u_t := u(s,t)$. 

We start with some side calculations, and put things together in the end.  Since $\int_X \psi(\eta) \omega_u^n =1$, the product rule for the Levi-Civita connection gives:
\begin{equation}\label{eq: sidecalc1}
\int_X \psi'(\eta)\nabla_{\partial_t u}\eta\, \omega_u^n =\partial_t \int_X \psi(\eta) \omega_u^n= 0.
\end{equation}
Using this identity and the product rule of the Levi-Civita connection again, we can differentiate the denominator of the right hand side of \eqref{eq: OrliczNormDiffEq} and obtain:
\begin{equation}\label{eq: sidecalc2}
\partial_t {\int_X \psi'(\eta) \eta \,  \omega_{u}^n} = {\int_X \psi''(\eta) \, \eta \, \nabla_{\partial_t u}(\eta) \omega_{u}^n}.
\end{equation}

Next we turn to the numerator of the right hand side of \eqref{eq: OrliczNormDiffEq}. The product rule again gives:
\begin{flalign}\label{eq: sidecalc3}
\partial_t  \int_X \psi' (\eta)\nabla_{\partial_t u} \partial_s u \, \omega_{u}^n &=\int_X \psi'' (\eta )\nabla_{\partial_t u} (\eta)  \nabla_{\partial_t u} \partial_s u  \, \omega_{u}^n + \int_X \psi' (\eta)\nabla_{\partial_t u}\nabla_{\partial_t u} \partial_s u \nonumber \\
&=\int_X \psi'' (\eta )\nabla_{\partial_t u} (\eta)  \nabla_{\partial_t u} \partial_s u  \, \omega_{u}^n + \int_X \psi' (\eta)\nabla_{\partial_t u}\nabla_{\partial_s u} \partial_t u \, \omega_{u}^n. 
\end{flalign}

For the last term on the right of \eqref{eq: sidecalc3} we make the following side computation:
\begin{flalign}\label{eq: sidecalc4}
\int_X&  \psi'(\eta) \nabla_{\partial_t u}\nabla_{\partial_t u} \partial_s u \omega_{u}^n = \int_X \psi'(\eta) R(\partial_t u, \partial_s u)\partial_t u  \, \omega_{u}^n + \int_X  \psi'(\eta) \nabla_{\partial_s u}\nabla_{\partial_t u} \partial_t u \, \omega_{u}^n \nonumber\\
&= \frac{1}{{\|\partial_s u\|_{\psi,u}}} \int_X \psi'' (\eta) \{\partial_s u,\partial_t u\}^2 \, \omega_u^n  + \varepsilon \int_X \psi' (\eta) \nabla_{\partial_s} \Big( \frac{\omega^n}{\omega_u^n}\Big) {\omega_u^n} \nonumber \\
&= \frac{1}{{\|\partial_s u\|_{\psi,u}}} \int_X \psi'' (\eta) \{\partial_s u,\partial_t u\}^2 \, \omega_u^n  - \nonumber\\
&-\varepsilon \int_X \psi' (\eta) \Big( \Delta_{\omega_u} \partial_s u \cdot \Big( \frac{\omega^n}{\omega_u^n}\Big) + \langle \nabla_{\omega_u}\Big( \frac{\omega^n}{\omega_u^n}\Big),\nabla_{\omega_u}\partial_s u  \rangle_{\omega_u}\Big) {\omega_u^n} \nonumber \\
&= \frac{1}{{\|\partial_s u\|_{\psi,u}}} \int_X \psi'' (\eta) \{\partial_s u,\partial_t u\}^2 \, \omega_u^n  + \frac{\varepsilon}{{\|\partial_s u\|_{\psi,u}}} \int_X \psi'' (\eta) \langle \nabla_{\omega_u} \partial_s u , \nabla_{\omega_u} \partial_s u \rangle_{\omega_u} {\omega^n},
\end{flalign}
where in the second we used the precise formula curvature $R(\cdot,\cdot)(\cdot)$ (computed in \cite[(5.13)]{CC3} or \cite[Theorem 5]{Bl13}) and \eqref{eq: epsgeodeq}, in the third line we used the formula for the Levi-Civita connection, and in the last line we used integration by parts.

For the first term of \eqref{eq: sidecalc3} we use that $\partial_s u =  \| \partial_s u \|_{\psi,u}\eta$, and the product rule for the Levi-Civita connection:
\begin{flalign}\label{eq: sidecalc5}
\int_X& \psi'' (\eta)\nabla_{\partial_t u}(\eta)   \nabla_{\partial_t u} \partial_s u \,  \omega_{u}^n = \\
& =\|\partial_s u\|_{\psi,u}  \int_X \psi'' (\eta) \big(\nabla_{\partial_t u}( \eta)\big)^2 \, \omega_{u}^n  +  \partial_t \|\partial_s u\|_{\psi,u}   \int_X \psi'' (\eta)  \eta \nabla_{\partial_t u}(\eta)  \, \omega_{u}^n \nonumber
\end{flalign}

Substituting \eqref{eq: sidecalc4} and \eqref{eq: sidecalc5} into \eqref{eq: sidecalc3} we arrive at

\begin{flalign}\label{eq: sidecalc6}
\partial_t  \int_X & \psi' (\eta)\nabla_{\partial_t u} \partial_s u \, \omega_{u}^n =   \int_X  \psi''(\eta)\Big(  \partial_t \|\partial_s u\|_{\psi,u}     \eta \nabla_{\partial_t u}(\eta) +   \|\partial_s u\|_{\psi,u} \big(\nabla_{\partial_t u}\eta \big)^2 + \nonumber \\
&+ \frac{1}{ \|\partial_s u\|_{\psi,u}}\{\partial_s u,\partial_t u\}^2 + \frac{\varepsilon}{ \|\partial_s u\|_{\psi,u}} \langle \nabla_{\omega_u} \partial_s u , \nabla_{\omega_u} \partial_s u \rangle \frac{\omega^n}{\omega_{u}^n}  \Big) \omega_{u}^n
\end{flalign}
 
Differentiating \eqref{eq: OrliczNormDiffEq}, we bring the above calculations together:

\begin{flalign}\label{eq: fullcalc}
\partial^2_t  \|\partial_s u\|_{\psi,u} &= \frac{\partial_t  \int_X \psi' (\eta)\nabla_{\partial_t u} \partial_s u \omega_{u}^n}{\int_X \psi'(\eta) \eta  \omega_{u}^n} - \partial_t  \|\partial_s u\|_{\psi,u} \frac{ \partial_t {\int_X \psi'(\eta) \eta  \omega_{u}^n}}{\int_X \psi'(\eta) \eta  \omega_{u}^n} \\
 &= \frac{\partial_t  \int_X \psi' (\eta)\nabla_{\partial_t u} \partial_s u \omega_{u}^n}{\int_X \psi'(\eta) \eta  \omega_{u}^n} - \partial_t  \|\partial_s u\|_{\psi,u} \frac{  {\int_X \psi''(\eta) \eta \nabla_{\partial_t u}(\eta) \omega_{u}^n}}{\int_X \psi'(\eta) \eta  \omega_{u}^n}, \nonumber
\end{flalign}
where in the last line we used \eqref{eq: sidecalc2}.
Next, in the numerator of the  first fraction we now substitute \eqref{eq: sidecalc6} and notice that the last term on the right hand side of  \eqref{eq: fullcalc} will cancel with the first term on the right hand side of \eqref{eq: sidecalc6}, ultimately yielding \eqref{eq: OrliczNormSecondDiffEq}.
\end{proof}

\section{The candidate metric $d_\psi$}

We start with a preliminary discussion of our candidate metric $d_\psi$, defined in \eqref{eq: d_psi_def_main}. By He's theorem \cite[Theorem 1.1]{He15}, we know that for $u_0,u_1 \in \mathcal H_\omega^{\Delta} := \textup{PSH}(X,\omega)\cap \{ \Delta_\omega v \in L^\infty\}=\textup{PSH}(X,\omega)\cap C^{1\bar 1}$, we also have $u_t \in \mathcal H_\omega^\Delta, \ t \in [0,1]$, where $t \to u_t$ is the weak geodesic connecting $u_0,u_1$. 

It is not yet known if $t \to u_t$ is $C^1$ in the $t$-direction when $u_0,u_1 \in \mathcal H_\o^{1,\bar 1}$. However, since $t \to u_t$ is $t$-convex, it makes sense to define $\dot u_0$ as the right derivative at $t =0$ and $\dot u_1$ as the left derivative at $t =1$. As a result, it is possible to extend the definition \eqref{eq: d_psi_def_main} to potentials with bounded Laplacian:
\begin{equation}\label{eq: d_psi_def}
d_\psi(u_0,u_1) := \int_X \psi(\dot u_0) \omega_{u_0}^n.
\end{equation}

This will be helpful since many operations on potentials are stable in the class $\mathcal H_\omega^\Delta$, and are not stable in $\mathcal H_\omega$. For example, by \cite[Theorem 2.5]{DR16}, we know that $u,v \in \mathcal H_\omega^\Delta$ implies $P(u,v) \in \mathcal H_\omega^\Delta$. The same property is not true for potentials of $\mathcal H_\omega$.

In addition, we also introduce
\begin{equation}\label{eq: d_psi_hat_def}
\hat d_\psi(u_0,u_1) =  \|\dot u_0\|_{\psi,u_0},
\end{equation}
where the term on the right hand side is the weak quasi-norm of $\dot u_0$ with respect to the weight $\psi$ and the measure $\omega_{u_0}^n$ \eqref{eq: quasi-norm-def}.

In case of $u_0,u_1 \in \mathcal H_\omega$, by \cite[Lemma 4.10]{Da15} (slightly extending \cite[Proposition 2.2]{Bern13a}) we obtain that $t \to \int_X \psi(c \dot u_t) \omega_{u_t}^n$ is constant for any $c \in \Bbb R_+$. By definition of the weak quasi-norms, this immediately gives that  $t \to \|\dot u_t\|_{\psi,u_t}$ is constant as well, hence in this case:
\begin{equation}\label{eq: d_psi_hat_def_length}
\hat d_\psi(u_0,u_1) = \|\dot u_t\|_{\psi,u_t}, \ \  \textup{ for any } t  \in [0,1].
\end{equation}
\begin{equation}\label{eq: d_psi_def_length}
d_\psi(u_0,u_1) =  \int_X \psi(\dot u_t) \omega_{u_t}^n,   \ \ \textup{ for any } t  \in [0,1].
\end{equation}
Though we will not need it, by \cite[Lemma 4.10]{Da15} the same formulas hold in case $u_0,u_1 \in \mathcal H^\Delta_\omega$ as well. However in this case one has to clarify what $\dot u_t$ means for $t \in (0,1)$. As follows from \cite[Lemma 4.10]{Da15} (and its proof), $\partial_t^+ u_t = \partial_t^- u_t$ a.e with respect to $\omega_{u_t}^n$. As a result, $\dot u_t$ is a.e. well defined with respect to $\omega_{u_t}^n$, allowing to make sense of the right hand side of  \eqref{eq: d_psi_hat_def_length} and \eqref{eq: d_psi_def_length} in this more general situation as well.

First we prove an approximation result for the above introduced notions:

\begin{proposition}\label{prop: pushforw_approx} Let $u_0,u_1 \in \mathcal H_\omega^\Delta$ and $u^k_0,u^k_1  \in \mathcal H^\Delta_\omega$ such that $u_0^k \to u_0$ and $u_1^k \to u_1$ uniformly. Then we have that $d_\psi(u^k_0,u^k_1) \to d_\psi(u_0,u_1)$ and $\hat d_\psi(u^k_0,u^k_1) \to \hat d_\psi(u_0,u_1)$.
\end{proposition}

\begin{proof} First we show that $d_\psi(u^k_0,u^k_1) \to d_\psi(u_0,u_1)$.
Let $[0,1] \ni t \to  u_t,u^k_t \in \mathcal H_\omega^\Delta$ be the weak geodesic joining $u_0,u_1$ and $u_0^k,u^k_1$ respectively. We first claim that the push-forward measures $ |{\dot {u}}^k_0|_* \omega_{u^k_0}^n$ weakly converge to $|\dot {u_0}|_* \omega_{u_0}^n$. 

Assuming the claim, since ${\dot {u}}^k_0, \dot {u_0}$ are uniformly bounded \cite[Theorem 1]{Da17}, we can apply this to $\psi$ to arrive at the conclusion:
$$d_\psi(u^k_0,u^k_1)  = \int_X \psi(\dot u^k_0) \omega_{u^k_0}^n=\int_X \psi(|\dot u^k_0|) \omega_{u^k_0}^n \to  \int_X \psi(|\dot u_0|) \omega^n_{u_0} = d_\psi(u_0,u_1).$$

Now we prove the claim. From \cite[Theorem 3]{Da15} and the triangle ineqality for $d_p$ we know that $d_p(u^k_0,u^k_1) \to d_p(u_0,u_1)$ for all $p \geq 1$. By \cite[Lemma 4.11]{Da15} this is equivalent with $\int_X |\dot u^k_0|^p \omega_{u^k_0}^n \to  \int_X |\dot u_0|^p \omega^n_{u_0}$. 

Since the global masses of the pushforward measures $ |\dot u^k_0|_* \omega_{u^k_0}^n, \  |\dot {u_0}|_* \omega_{u^k_0}^n$ are finite, and $\dot u^k_0,\dot u_0$ are uniformly bounded, the Stone-Weierstrass theorem implies that $\int_X \alpha(|\dot u^k_0|) \omega_{u_0^k}^n \to \int_X \alpha(|\dot u_0|) \omega_{u_0}^n$ for any $\alpha \in C(\Bbb R)$. This is equivalent with $ |\dot u^k_0|_* \omega_{u^k_0}^n \to |\dot {u_0}|_* \omega_{u_0}^n$, as desired.

We can repeat the above for $\psi(c t)$ instead of $\psi(t)$ for any $c \in [0,\infty)$, and conclude that $\hat d_\psi(u^k_0,u^k_1) \to \hat d_\psi(u_0,u_1)$ via Lemma \ref{lem: convergence}. 
\end{proof}

Next, we point out that an analogue of the Pythagorean identity holds for $d_\psi$:

\begin{lemma}\label{lem: Pythagorean} Let $\psi \in \mathcal W^-$ and $u,v \in \mathcal H_\omega^\Delta$. Then for $P(u,v) \in \mathcal H_\omega^\Delta$ we have that
$$d_\psi(u,v) = d_\psi(u,P(u,v)) + d_\psi(v,P(u,v)).$$
\end{lemma}
\begin{proof} This is a consequence of \cite[Proposition 4.13]{Da15} for $f := \psi$.
\end{proof}

Next we point out that the operator $u \to P(u,w)$ is $d_\psi$-shrinking:

\begin{proposition} \label{prop: contractive}  Let $\psi \in \mathcal W_-$ and $u,v,w \in \mathcal H_\omega^{\Delta}$. Then we have
$$d_\psi(P(u,w),P(v,w)) \leq d_\psi(u,v).$$
\end{proposition}

\begin{proof} The proof is exactly the same as that of \cite[Proposition 8.2]{Da17b}, where one replaces the convex weight $|t|^2$ with our weight $\psi \in \mathcal W^-$.
\end{proof}

Using the argument of \cite[Lemma 4.2]{Da15} we note the following lemma:

\begin{lemma}\label{lem: mon_triplet_ineq} Let $\alpha,\beta,\gamma \in {\mathcal H}_\omega^\Delta$ such that $\alpha \geq \beta \geq \gamma$. Then $d_\psi(\alpha,\beta) \leq d_\psi(\alpha,\gamma)$ and $\hat d_\psi(\alpha,\beta) \leq \hat d_\psi(\alpha,\gamma)$. Analogously, $d_\psi(\gamma,\beta) \leq d_\psi(\gamma,\alpha)$ and $\hat d_\psi(\gamma,\beta) \leq \hat d_\psi(\gamma,\alpha)$.
\end{lemma}
\begin{proof} Let $[0, 1] \ni t \to u_t,v_t \in \mathcal H_\omega^\Delta$ be the weak geodesics connecting $\alpha, \beta$ and $\alpha,\gamma$ respectively. We notice that they are both decreasing, satisfy $u_t \geq v_t$ by the comparison principle, and $u_0 = v_0=\alpha$. From this it follows that $0 \geq \dot u_0 \geq \dot v_0$. Using this, \eqref{eq: d_psi_def} and \eqref{eq: d_psi_hat_def} yield that $d_\psi(\alpha,\beta) \leq d_\psi(\alpha,\gamma)$ and $\hat d_\psi(\alpha,\beta) \leq \hat d_\psi(\alpha,\gamma)$. The last sentence is proved analogously, using two weak geodesics meeting at $\gamma$.
\end{proof}

\section{The triangle inequality}

First we obtain the triangle inequality for $d_\psi$ in a special case (Proposition \ref{prop: concavity}), and then derive the general version from this using the Pythagorean identity for $d_\psi$.

We start with the analogue of \cite[Lemma 5.2]{CC3} in our setting, that will only hold in the particular case of increasing smooth curves (c.f. \cite[Theorem 1.2]{Le20}).

\begin{proposition}\label{prop: lengthconcave} Let $\varepsilon>0$ and $[0,1] \ni s \to u_{0,s},u_{1,s} \in \mathcal H_\omega$ be smooth curves satisfying $\partial_s u_{0,s}>0,\partial_s u_{1,s}>0$. For fixed $s$, let $[0,1] \ni t \to u^\varepsilon_{t,s}\in \mathcal H_\omega$ be the $\varepsilon$-geodesic connecting $u_{0,s},u_{1,s}$. Then $t \to \int_0^1 \|\partial_s u^\varepsilon_{t,s}\|_{\psi,u^\varepsilon_{t,s}}ds $ is concave.
\end{proposition}

\begin{proof} By assumption, $\partial_s u_{0,s},\partial_s u_{1,s} \geq \delta>0$ for some constant $\delta$. By the proof of \cite[Corollary 3.4]{Da19} we get that $\partial_s u^\varepsilon_{t,s} \geq \delta>0$ for any $t,s \in [0,1]$  and $\varepsilon >0$. In particular, $ s \to \| \partial_s u \|_{\psi, u}$ is  smooth and the results of Section 2.3 are applicable. In particular, since $\psi|_{[0,\infty)}$ is concave, Proposition \ref{prop: OrliczNormSecondDiff} gives:

$$\frac{d^2}{dt^2 }\int_0^1 \|\partial_s u^\varepsilon_{t,s}\|_{\psi,u_{t,s}}  ds = \int_0^1 \frac{d^2}{dt^2 }  \|\partial_s u^\varepsilon_{t,s}\|_{\psi,u_{t,s}}  ds \leq 0.$$
This is equivalent to concavity of $t \to \int_0^1 \|\partial_s u^\varepsilon_{t,s}\|_{\psi,u^\varepsilon_{t,s}} ds $.
\end{proof}

\begin{proposition}\label{prop: maxilengthgeod}Let $u_0,u_1 \in \mathcal H_\omega$ with $u_0 < u_1$. Let $[0,1] \ni s \to u_s \in \mathcal H_\omega^\Delta$ be the (increasing) weak geodesic joining $u_0,u_1$. Then $\hat d_\psi(u_0,u_1)=\int_0^1  \|\dot u_s\|_{\psi,u_s}   ds \geq \int_0^1 \|\dot \zeta_s \|_{\psi,\zeta_s}  ds$, where $t \to \zeta_t$ is any smooth increasing curve $(\dot \zeta_s >0)$ joining $\zeta_0 := u_0$ and $\zeta_1 := u_1$.
\end{proposition}

The following argument is due to L. Lempert.

\begin{proof}  Let $\delta>0$ such that $u_1 - u_0 > \delta$ and $\dot \zeta_s >\delta$ for all $s \in [0,1]$. 
From \eqref{udef1} we obtain that $u_t \geq u_0 + \delta t$. Since $t \to u_t$ is $t$-convex, we obtain that $\dot u_t \geq \dot u_0 \geq \delta$.

By \eqref{eq: d_psi_hat_def_length} we know that $s \to \| \dot u_s\|_{\psi,u_s}$ is constant equal to $c>0$. Since $\psi$ is concave and smooth on $({\delta}/{2c},\infty)$, it admits a concave extension $\tilde \psi$ to $(-\infty,\infty)$ such that $\tilde \psi|_{({\delta}/{2c},\infty)} = \psi|_{({\delta}/{2c},\infty)}$. Such extension of course is non-unique.

As $\dot \zeta_s>0$, Proposition \ref{prop: OrliczNormDiff} implies that $s \to \| \dot \zeta_s \|_{\psi,\zeta_s}$ is smooth. Since weak quasi-norms are homogeneous, it is possible to reparametrize $[0,1] \ni s \to \zeta_s \in \mathcal H_\omega$ to a smooth curve $[0,1] \ni s \to \tilde \zeta_s \in \mathcal H_\omega$ such that $s \to \| \dot {\tilde \zeta}_s \|_{\psi,\zeta_s}$ is constant and the $\psi$-arclength does not change:

$$\int_0^1 \|\dot \zeta_s \|_{\psi,\zeta_s}  ds=\int_0^1 \|\dot { \tilde \zeta}_s \|_{\psi,\tilde \zeta_s}  ds=\|\dot {\tilde \zeta}_l \|_{\psi,\tilde \zeta_l}, \ \ l \in [0,1].$$
Since $t \to -\tilde \psi(t/c)$ is convex, \cite[Theorem 1.1]{Le20} implies that 
\begin{flalign*}
1 = \int_0^1\int_X \psi\Big(\frac{\dot u_s}{c}\Big) \omega_{u_s}^n ds  &= \int_0^1\int_X \tilde \psi\Big(\frac{\dot u_s}{c}\Big) \omega_{u_s}^n ds \\
&\geq \int_0^1\int_X \tilde \psi \Big(\frac{\dot {\tilde \zeta}_s}{c} \Big) \omega_{\zeta_s}^n ds=\int_0^1\int_X  \psi \Big(\frac{\dot {\tilde \zeta}_s}{c} \Big) \omega_{\zeta_s}^n ds.
\end{flalign*}

In particular, by the mean value theorem, we obtain that $\int_X  \psi(\frac{\dot {\tilde \zeta}_t}{c} ) \omega_{\tilde \zeta_t}^n  \leq 1$ for some $t \in [0,1]$. By the definition of the weak quasi-norm, we get that $\|{\dot {\tilde \zeta}}_t \|_{\psi,\tilde{\zeta}_t} \leq c.$ But since $s \to \|\dot {\tilde{\zeta}}_s \|_{\psi,\tilde{\zeta}_s}$ is constant, we actually get that $\|\dot {\tilde{\zeta}}_s \|_{\psi,\tilde{\zeta}_s} \leq c=\| \dot u_s\|_{\psi,u_s}$ for all $s \in [0,1]$. Integrating this inequality on $[0,1]$ yields the desired estimate.
\end{proof}

As a corollary of the above two results, we obtain the following:

\begin{corollary}\label{cor: concavity} Suppose we are given $\alpha, \beta, \gamma \in \mathcal H_\omega^\Delta$ such that $\alpha \geq  \beta \geq \gamma$. Let $[0,1] \ni t \to \alpha_t,\gamma_t \in \mathcal H_\omega^\Delta$ be the weak geodesic joining $\alpha_0:= \beta,\alpha_1:= \alpha$ and $\gamma_0:= \beta,\gamma_1:=\gamma$
 respectively. Then the function $t \to \hat d_\psi(\alpha_t,\gamma_t)$ is concave.
\end{corollary}

\begin{proof} Using Proposition \ref{prop: pushforw_approx} and \cite{De94}, we can assume that $\alpha,\beta,\gamma \in \mathcal H_\omega$ and $\alpha> \beta > \gamma$. 

As in the first step of the proof of Proposition \ref{prop: maxilengthgeod}, there exists $\delta >0$ such that $\dot \alpha_t > \delta$ and  $\dot \gamma_t <-\delta$ for all $t \in [0,1]$. 
Let $[0,1]  \ni t \to \alpha^\varepsilon_t,\gamma^\varepsilon_t \in \mathcal H_\omega$ be the $\varepsilon$-geodesic joining $\alpha^\varepsilon_0:= \beta,\alpha^\varepsilon_1:= \alpha$ and $\gamma^\varepsilon_0:= \beta,\gamma^\varepsilon_1:=\gamma$
 respectively. As $\varepsilon-$geodesics converge to weak geodesics in the $C^{1,\alpha}$-topology, for small enough $\varepsilon$, we also have $\dot \alpha^\varepsilon_t>\delta$ and $ \dot \gamma^\varepsilon_t <-\delta$. In particular, $t \to \alpha^\varepsilon_t$ is strictly increasing and $t \to \gamma^\varepsilon_t$ is strictly decreasing

Let $t,t' \in (0,1]$. Now let $\varepsilon'>0$ be small enough so that both $\varepsilon'$-geodesics $[0,1] \ni s \to v^{\varepsilon,\varepsilon',t}_s,v^{\varepsilon,\varepsilon',t'}_s \in \mathcal H_\omega$ joining $v^{\varepsilon,\varepsilon',t}_0 := \gamma^\varepsilon_t$ and $v^{\varepsilon,\varepsilon',t}_1 := \alpha^\varepsilon_t$, respectively $v^{\varepsilon,\varepsilon',t'}_0 := \gamma^\varepsilon_{t'}$ and $v^{\varepsilon,\varepsilon',t'}_1 := \alpha^\varepsilon_{t'}$ are strictly increasing (i.e. $\partial_s v^{\varepsilon,\varepsilon',t}_s,\partial_s v^{\varepsilon,\varepsilon',t'}_s>0$).

For $s \in [0,1]$ fixed, let  $[0,1] \ni \lambda \to \eta^{\varepsilon',\varepsilon}(\lambda,s)$ be the $\varepsilon$-geodesic joining $\eta^{\varepsilon',\varepsilon}(0,s):=v^{\varepsilon,\varepsilon',t}_s$ and $\eta^{\varepsilon,\varepsilon',\varepsilon}(1,s):=v^{\varepsilon,\varepsilon',t'}_s$. Notice that $\eta^{\varepsilon',\varepsilon}(\lambda,1) = \alpha^\varepsilon_{(1-\lambda) t + \lambda t'} $ and $\eta^{\varepsilon',\varepsilon}(\lambda,0) = \gamma^\varepsilon_{(1-\lambda) t + \lambda t'}$.

We fix $\lambda \in [0,1]$. Combining previous results we can finish the proof:
\begin{flalign*}
(1-\lambda)& \hat d_\psi(\alpha_t,\gamma_t) + \lambda \hat d_\psi(\alpha_{t'},\gamma_{t'}) = \lim_{\varepsilon \to 0 } \big( (1-\lambda)  \hat d_\psi (\alpha^\varepsilon_t,\gamma^\varepsilon_t) + \lambda \hat d_\psi(\alpha^\varepsilon_{t'},\gamma^\varepsilon_{t'}) \big)\\
&=\lim_{\varepsilon \to 0} \lim_{\varepsilon' \to 0}\bigg(\int_0^1 \Big( (1-\lambda) \|\partial_s v^{\varepsilon,\varepsilon',t}\|_{\psi,v^{\varepsilon,\varepsilon'}}    \omega_{v^{\varepsilon,\varepsilon',t}}^n+ \lambda \|\partial_s v^{{\varepsilon,\varepsilon'},t'}\|_{\psi,v^{\varepsilon,\varepsilon'}} \Big) ds\bigg) \\
&\leq \lim_{\varepsilon \to 0} \int_0^1  \|\partial_s \eta^{\varepsilon',\varepsilon}((1-\lambda) t + \lambda t',s)\|_{\psi, \eta^{\varepsilon',\varepsilon}((1-\lambda) t + \lambda t',s)}     ds \\
&\leq  \lim_{\varepsilon \to 0} \hat d_\psi(\alpha^\varepsilon_{(1-\lambda)t + \lambda t'},\gamma^\varepsilon_{(1-\lambda)t + \lambda t'}) \\
&= \hat d_\psi(\alpha_{(1-\lambda)t + \lambda t'},\gamma_{(1-\lambda)t + \lambda t'}),
\end{flalign*}
where in the first line we have used Proposition \ref{prop: pushforw_approx}, in the second line we have used \eqref{eq: d_psi_hat_def_length} and Lemma \ref{lem: convergence2}, in the third line we have used Proposition \ref{prop: lengthconcave}, in the fourth line we have used Proposition \ref{prop: maxilengthgeod}, and in the last line we have used Proposition \ref{prop: pushforw_approx} again.
\end{proof}

\begin{proposition}\label{prop: concavity} Given $\alpha, \beta, \gamma \in \mathcal H_\omega^\Delta$ such that $\alpha \geq \beta \geq \gamma$, we have that 

\begin{equation}\label{eq: triangle_ineq_special}
d_\psi(\alpha,\gamma) \leq d_\psi(\alpha,\beta) + d_\psi(\beta,\gamma).
\end{equation}

\end{proposition}

\begin{proof} Using Proposition \ref{prop: pushforw_approx}, we can assume that $\alpha,\beta,\gamma$ are smooth K\"ahler potentials, moreover $\alpha > \beta> \gamma$.

Let $[0,1] \ni t \to u_t,v_t \in \mathcal H_\omega^\Delta$ be the weak geodesics connecting $u_0: = \beta$ and $u_1 := \alpha$, respectively $v_0: = \beta$ and $v_1 := \gamma$. By Corollary \ref{cor: concavity}  we get that $t \to \hat d_\psi(u_t,v_t)$ is concave. Hence, since $\hat d_\psi(u_0,v_0) = 0$, $t \to \hat d_\psi(u_t,v_t)/t$ is decreasing.

Let $[0,1] \ni t \to \eta_t$ be the weak geodesic connecting $\eta_0 = \alpha$ and $\eta_1 = \gamma$. We can use the $\psi$-version of \cite[Lemma 4.1]{Da15} (whose proof is identical) to write:
\begin{flalign*}
\|\dot \eta_0  \|_{\psi,\alpha} &= \hat d_\psi(\alpha,\gamma)  = \hat d_\psi(u_1,v_1)\leq \lim_{t \to 0} \frac{\hat d_\psi(u_t,v_t)}t \leq  \limsup_{t \to 0} \frac{\|u_t - v_t\|_{\psi,v_t}}{t} =\\ 
& = \limsup_{t \to 0} \Big\| \frac{u_t - v_t}{t}\Big\|_{\psi,v_t}  =  \|  \dot u_0 - \dot v_0\|_{\psi,\beta}, 
\end{flalign*}
where  in the last step we have used that $\omega_{v_t}^n \to \omega_{\beta}^n$ weakly, moreover $(u_t-v_t)/t \to \dot u_0 - \dot v_0$ uniformly, as $t \to 0$. Indeed, this allows an application of  Lemma \ref{lem: convergence2} to conclude.

Finally, if we replace $\psi(t)$ with $\tilde \psi(t) := \psi(t) / \int_X \psi(\dot \eta_0) \omega_\alpha^n \in \mathcal W^-$, the same inequality as above implies that, $1 = \|\dot \eta_0  \|_{\tilde \psi,\alpha} \leq \|\dot u_0 - \dot v_0  \|_{\tilde \psi,\beta} $, i.e., $\int_X \tilde \psi(\dot u_0 - \dot v_0) \omega_\beta^n \geq 1$, i.e.,  $\int_X {\psi(\dot u_0 - \dot v_0)} \omega_\beta^n \geq \int_X \psi(\dot \eta_0) \omega_\alpha^n$. Using Lemma \ref{lem: sublinear} we now conclude that 
\begin{flalign*}
d_\psi(\alpha,\beta) + d_\psi(\beta,\gamma) &= \int_X \psi(\dot u_0) \omega_\beta^n + \int_X \psi(\dot v_0) \omega_\beta^n \\
&\geq \int_X \psi(\dot u_0-\dot v_0) \omega_\beta^n\geq \int_X \psi(\dot \eta_0) \omega_\alpha^n = d_\psi(\alpha,\gamma).
\end{flalign*}
\end{proof}

We are ready to prove the general case of the triangle inequality.

\begin{theorem}\label{thm: triangle_ineq} Let $u,v,w \in \mathcal H_\omega^\Delta$. Then we have
\begin{equation}\label{eq: triangle_ineq}
d_\psi(u,w) \leq d_\psi(u,v) + d_\psi(v,w).
\end{equation}
\end{theorem}
\begin{proof} The triangle inequality follows from the following sequence of inequalities:
\begin{flalign*}
&d_\psi(u,v) + d_\psi(v,w)  = d_\psi(u,P(u,v)) + d_\psi(P(u,v),v) + d_\psi(v,P(v,w)) + d_\psi(P(v,w),w)  \\
& \geq d_\psi(u,P(u,v)) + d_\psi(P(v,w),P(u,v,w)) + d_\psi(P(u,v),P(u,v,w)) + d_\psi(P(v,w),w)\\
& = d_\psi(u,P(u,v)) + d_\psi(P(u,v),P(u,v,w))  + d_\psi(w, P(v,w)) + d_\psi(P(v,w),P(u,v,w))\\
& \geq  d_\psi(u,P(u,v,w))  + d_\psi(w, P(u,v,w))\\
& \geq d_\psi(u,P(u,w))  + d_\psi(w, P(u,w))\\
&= d_\psi(u,w),
\end{flalign*}
where in the first and last line we have used the Pythagorean identity for $d_\psi$ (Lemma \ref{lem: Pythagorean}), in the second line we have used Proposition \ref{prop: contractive} for the second and third terms, in the fourth line we have used twice the particular case of the triangle inequality obtained in  Proposition \ref{prop: concavity}, and  in the fifth line we used Lemma \ref{lem: mon_triplet_ineq}.
\end{proof}

\begin{corollary}\label{cor: main_cor} $(\mathcal H^\Delta_\omega,d_\psi)$ is a metric space.
\end{corollary}
\begin{proof} By the previous result, we only need to argue that $d_\psi(u_0,u_1)=0$ implies $u_0 = u_1$. 

If $d_\psi(u_0,u_1)=0$, by Lemma \ref{lem: Pythagorean} we have $d_\psi(u_0,P(u_0,u_1))=0$ and also $d_\psi(u_1,P(u_0,u_1))=0$. By the first estimate of Proposition \ref{prop: Mdist_est} below, it follows that $u_0=P(u_0,u_1)$ a.e. with respect to $\o_{P(u_0,u_1)}^n$, and similarly, $u_1= P(u_0,u_1)$ a.e. with respect to $\o_{P(u_0,u_1)}^n$. We can now use  the domination principle of full mass potentials due to Dinew \cite[Proposition 5.9]{BL12} to obtain that $u_0 \leq P(u_0,u_1)$ and $u_1 \leq P(u_0,u_1)$. As the reverse inequalities are trivial, we get that $u_0 = P(u_0,u_1)=u_1$. 
\end{proof}

\begin{proposition}\label{prop: Mdist_est}Suppose $u,v \in \mathcal H^\Delta_\o$ with $u \leq v$. Then we have:
\begin{equation}\label{eq: Mdist_est}
\max\Big( \frac{1}{2^{n+1}}\int_X \psi(v-u) \o_u^n, \int_X \psi(v-u) \o_v^n \Big) \leq d_\psi(u,v) \leq \int_X\psi(v-u) \o_u^n.
\end{equation}
\end{proposition}

\begin{proof} Using Proposition \ref{prop: pushforw_approx} we can assume that $u$ and $v$ are smooth. Suppose $[0,1] \ni t \to w_t \in \mathcal H^\Delta_\o$ is the weak geodesic segment joining $w_0 = u$ and $w_1 = v$. By \eqref{eq: d_psi_def_length} we have
$$d_\psi(u,v)=\int_X \psi(\dot w_0) \o_u^n=\int_X \psi(\dot w_1) \o_v^n.$$ 
Since $u \leq v$, we have that $u \leq w_t$, as follows from the comparison principle. Since $(t,x) \to w_t(x)$ is convex in the $t$ variable, we get $0 \leq \dot w_0 \leq v-u \leq \dot w_1$, and together with the above identity we obtain part of \eqref{eq: Mdist_est}:
\begin{equation}\label{eq: Mdist_est_interm}
\int_X \psi(v-u) \o_v^n \leq d_\psi(u,v) \leq \int_X \psi(v-u) \o_u^n.
\end{equation}
Now we prove the rest of \eqref{eq: Mdist_est}. Using $ \o_{u}^n \leq 2^{n}\o_{(u+v)/2}^n$  and concavity of $\psi$ on $[0, \infty),$ we obtain that 
$$\frac{1}{2^{n+1}}\int_X \psi(v - u) \o_{u}^n \leq \int_X \psi\Big( \frac{u+v}{2} - u\Big)\o^n_{(u+v)/2}.$$
Since $u \leq (u+v)/2$, the first estimate of \eqref{eq: Mdist_est_interm} allows to write:
$$\frac{1}{2^{n+1}}\int_X \psi(v-u) \o_{u}^n \leq d_\psi \Big(\frac{u + v}{2},u\Big).$$
Finally, Lemma \ref{lem: mon_triplet_ineq} implies that $d_\psi((u + v)/2,u) \leq d_\psi(v,u)$, giving the remaining estimate in \eqref{eq: Mdist_est}.
\end{proof}

\section{Extending $d_\psi$ to $\mathcal E_\psi$ and completeness}

Given $u_0,u_1 \in \mathcal E_\psi(X,\o)$, by a classical result of Demailly \cite{De94} (see \cite{BK07} for a short argument) there exists decreasing sequences $u^k_0, u^k_1 \in \mathcal H_\o$ such that $u^k_0 \searrow u_0$ and $u^k_1 \searrow u_1$. We propose to extend $d_\psi$ to $\mathcal E_\psi$ in the following way:
\begin{equation}\label{eq: dpsi_def_general}
d_\psi(u_0,u_1)=\lim_{k\to \infty}d_\psi(u^k_0,u^k_1).
\end{equation}
Very similar to the high energy case \cite{Da15}, we will show that the limit on the right hand side exists and is independent of the approximating sequences. For this, we first prove the next lemma:

\begin{lemma} Suppose $u \in \mathcal E_\psi$ and $\{ u_k\}_{k} \subset \mathcal H^\Delta_\o$ is a sequence decreasing to $u$. Then $d_\psi(u_l,u_k) \to 0$ as $l,k \to \infty$.\label{lem: IntDistEst}
\end{lemma}

\begin{proof}
We can suppose that $l \leq k$. Then $u_k\leq u_l$, hence by Proposition \ref{prop: Mdist_est} we have:
$$d_\psi(u_l,u_k) \leq \int_X\psi(u_k-u_l)\o_{ u_k}^n.$$
Let us fix $l$ momentarily, and let $\{v_j\}_j \in \mathcal H_\omega$ be such that $v_j \searrow u_l$. Then $u - v_j, u_k - v_j \in \mathcal E_\psi(X,\omega_{v_j})$ and $u - v_j\leq u_k - v_j\leq0$. Hence, applying Proposition \ref{prop: Energy_est}  for the class $\mathcal E_\psi(X,\omega_{u_l})$ we obtain
\begin{flalign}\label{eq: estimate}
d_\psi(u_l,u_k)&\leq \int_X\psi(u_k-u_l)\o_{ u_k}^n \leq \lim_j \int_X\psi(u_k-v_j)\o_{u_k}^n \nonumber\\
& \leq C \lim_j \int_X\psi(u-v_j)\o_u^n = C \int_X\psi(u-u_l)\o_u^n.
\end{flalign}
As $u_l$ decreases to $u \in \mathcal E_\psi$, by the dominated convergence theorem we have $d_\psi(u_l,u_k) \to 0$ as $l,k \to \infty$.
\end{proof}

Our next lemma confirms that the way we proposed to extend the $d_\psi$ metric to $\mathcal E_\psi$  in \eqref{eq: dpsi_def_general} is consistent.

\begin{lemma} Given $u_0,u_1 \in \mathcal E_\psi$, the limit in \eqref{eq: dpsi_def_general} is finite and independent of the approximating sequences $u^k_0, u^k_1 \in \mathcal H^\Delta_\o$. \end{lemma}

\begin{proof} By Proposition \ref{prop: pushforw_approx} we can assume that the approximating sequences are smooth. By the triangle inequality and Lemma \ref{lem: IntDistEst} we can write:
$$|d_\psi(u^l_0,u^l_1)-d_\psi(u^k_0,u^k_1)| \leq d_\psi(u^l_0,u^k_0) +d_\psi(u^l_1,u^k_1) \to 0, \ l,k \to \infty, $$
proving that $d_\psi(u^k_0,u^k_1)$ is indeed convergent.

Now we prove that the limit in \eqref{eq: dpsi_def_general} is independent of the choice of approximating sequences. Let $v^l_0, v^l_1 \in \mathcal H_\o$ be different approximating sequences. By adding small constants we arrange that the sequences $u^l_0, u^l_1$, respectively $v^l_0, v^l_1$, are strictly decreasing to $u_0,u_1$.

Fixing $k$ for the moment, the sequence $\{\max\{ u^{k+1}_0,v^j_0\}\}_{j \in \Bbb N}$ decreases pointwise to $u^{k+1}_0$. By Dini's lemma there exists $j_k\in \Bbb N$ such that for any $j \geq j_k$ we have $v^j_0 < u^k_0$. By repeating the same argument we can also assume that $v^j_1 < u^k_1$ for any $j \geq j_k$. By the triangle inequality again
$$|d_\psi(u^k_0,u^k_1)-d_\psi(v^j_0,v^j_1)| \leq d_\psi(u^k_0,v^j_0) +d_\psi(u^k_1,v^j_1), \ j \geq j_k. $$
From \eqref{eq: estimate} it follows that for $k$ big enough $d_\psi(u^j_0,v^k_0)$, $d_\psi(u^j_1,v^k_1), \ j \geq j_k$ are arbitrarily small. As a result, $d_\psi(u_0,u_1)$ is independent of the choice of approximating sequences. 

When $u_0,u_1 \in \mathcal H^\Delta_\o$, one can approximate with the constant sequence, hence the restriction to $\mathcal H_\o$  of the extended $d_\psi$ from  \eqref{eq: dpsi_def_general} coincides with the original definition \eqref{eq: d_psi_def}. 
\end{proof}

By the above result, \cite[Proposition 2.20]{Da19}, and the remark following it,  many properties of $d_\psi$ extend to $\mathcal E_\psi$, in particular the triangle inequality, the Pythagorean formula,  etc. We list these in the proposition below and leave the standard proofs to the interested reader.

\begin{proposition}\label{prop: Epsi_extension} Let $\psi \in \mathcal W^-$. Then the following hold:\\
\noindent (i) $d_\psi:\mathcal E_\psi \times \mathcal E_\psi \to \Bbb R$ satisfies the triangle inequality. \\
\noindent (ii)  If $u,v \in \mathcal E_\psi$ then $P(u,v) \in \mathcal E_\psi$ and $d_\psi(u,v) = d_\psi(u,P(u,v)) + d_\psi(v,P(u,v)).$\\
\noindent (iii) Suppose $u,v \in \mathcal E_\psi$ with $u \leq v$. Then we have:
$$
\max\Big( \frac{1}{2^{n+1}}\int_X \psi(v-u) \o_u^n, \int_X \psi(v-u) \o_v^n \Big) \leq d_\psi(u,v) \leq \int_X\psi(v-u) \o_u^n.$$
\noindent (iv) For  $u,v,w \in \mathcal E_\psi$ we have
$d_\psi(P(u,w),P(v,w)) \leq d_\psi(u,v).$
\end{proposition}

We now argue that non-degeneracy of $d_\psi$ on $\mathcal E_\psi$ holds as well:

\begin{proposition}\label{prop: dpsi_nondegeneracy} Given $u_0,u_1 \in \mathcal E_\psi$ if $d_\psi(u_0,u_1)=0$ then $u_0=u_1$. In particular, $(\mathcal E_\psi,d_\psi)$ is a metric space.
\end{proposition}

\begin{proof} We can repeat the argument of Corollary \ref{cor: main_cor}. By Proposition \ref{prop: Epsi_extension}(ii) it follows that $d_\psi(u_0,P(u_0,u_1))=0$ and also $d_\psi(u_1,P(u_0,u_1))=0$. By Proposition \ref{prop: Epsi_extension}(iii), it follows that $u_0=P(u_0,u_1)$ a.e. with respect to $\o_{P(u_0,u_1)}^n$, and similarly, $u_1= P(u_0,u_1)$ a.e. with respect to $\o_{P(u_0,u_1)}^n$. We can now use  the domination principle of full mass potentials due to Dinew \cite[Proposition 5.9]{BL12} to obtain that $u_0 \leq P(u_0,u_1)$ and $u_1 \leq P(u_0,u_1)$. As the reverse inequalities are trivial, we get that $u_0 = P(u_0,u_1)=u_1$. 
\end{proof}

\begin{corollary} \label{cor: d_psi_monotone_limit}
If $\{w_k\}_{k \in \Bbb N} \subset \mathcal E_\psi$ decreases or increases a.e. to $w \in \mathcal E_\psi$ then $d_\psi(w_k,w)\to 0$.
\end{corollary}

\begin{proof} By Proposition \ref{prop: Epsi_extension}(iii), we have $d_\psi(w,w_k) \leq \int_X \psi(w -w_k) (\o_{w_k}^n + \o_{w}^n)$. We can use \cite[Proposition 2.20]{Da19} (and the remark following it) to conclude that  $d_\psi(w,w_k) \to 0$.
\end{proof}

\begin{lemma} Suppose $\{u_k\}_{k \in \Bbb N} \subset \mathcal E_\psi$ be an  increasing  $d_\psi$--bounded sequence. Then $\sup_X u_k$ is a bounded sequence.
\end{lemma}

\begin{proof}
Using Theorem \ref{thm: Energy_Metric_Eqv} from below we have that $\int_X \psi(\max(u_k,0)) \omega^n  \leq \int_X \psi(u_k) \omega^n \leq 2^{2n+5} d_\psi(u_k,0) \leq C,$ for some $C>0$. Let $v : = \lim_k \max(0,u_k))$, a  measurable function on $X$.  By the monotone convergence theorem we obtain that $\int_X \psi(v) \omega^n \leq C$.  
This implies that for some $d >0$ the set $K := \{v \leq d\}$ has non-zero Lebesgue measure, hence $K$ is also non-pluripolar.

On $K$ we have that $u_k \leq d$. As a result, due to \cite[Corollary 4.3]{GZ05} we obtain that $\{u_k\}_k \subset \textup{PSH}(X,\omega)$ is relatively $L^1$-compact, hence $\sup_X u_k$ can not converge to $\infty$, finishing the proof.
\end{proof}

Next we argue that bounded monotone sequences in $\mathcal E_\psi$ have limits inside $\mathcal E_\psi.$ Using the previous lemma, the proof of this result is very similar to \cite[Lemma 3.34]{Da19}:

\begin{lemma}\label{lem: lemma mononton_seq} Suppose $\{u_k\}_{k \in \Bbb N} \subset \mathcal E_\psi$ is a decreasing/increasing  $d_\psi$--bounded sequence. Then $u = \lim_{k \to \infty} u_k \in \mathcal E_\psi$ and additionally $d_\psi(u,u_k) \to 0$.
\end{lemma}

\begin{proof}Due to the previous lemma, after subtracting a constant, we can assume without loss of generality that $u_k \leq 0$.

Let us assume that $\{u_k\}_k$ is decreasing. From Proposition \ref{prop: Epsi_extension}(iii) we have that $\int_X \psi(u_k) \omega^n_{u_k}$ is uniformly bounded. Due to \cite[Proposition 5.6]{GZ07} we get that $u = \lim_k u_k \in \mathcal E_\psi(X,\omega)$. Corollary \ref{cor: d_psi_monotone_limit}  implies that $d_\psi(u_k,u)  \to 0$.

Now us assume that $\{u_k\}_k$ is increasing. Due to the previous lemma, there exists $u \in \mathcal E_\psi(X,\omega)$ such that $u_k \nearrow u$. By Corollary \ref{cor: d_psi_monotone_limit} again, $d_\psi(u_k,u) \to 0$.\end{proof}

Finally, we argue completeness of $(\mathcal E_\psi,d_\psi)$:

\begin{theorem} \label{thm: EpsiComplete}  $(\mathcal E_\psi,d_\psi)$ is a complete metric space, that is the metric completion of $(\mathcal H_\o,d_\psi)$. 
\end{theorem}

\begin{proof} By  Corollary \ref{cor: d_psi_monotone_limit}  and \cite{De94} $\mathcal H_\o$ is a $d_\psi$--dense subset of $\mathcal E_\psi$. We need to argue completeness, which can be done identically as in \cite{Da15}, due to Proposition \ref{prop: contractive}. 

Indeed, suppose $\{u_k\}_{k \in \Bbb N} \subset \mathcal E_\psi$ is a $d_\psi$--Cauchy sequence. We will prove that there exists $v \in \mathcal E_\psi$ such that $d_\psi (u_k,v) \to 0.$ After passing to a subsequence we can assume that
$$d_\psi(u_l,u_{l+1}) \leq 1/2^l, \ l \in \Bbb N.$$
By \cite[Theorem 3.6]{Da17b} we can introduce $v^k_l = P(u_k,u_{k+1},\ldots,u_{k+l}) \in \mathcal E_\psi, \ l,k \in \Bbb N$. We argue first that each decreasing sequence $\{ v^k_l\}_{l \in \Bbb N}$ is $d_\psi$--Cauchy. 
We observe that $v^k_{l+1}=P(v^k_l,u_{k+ l+1})$ and $v^k_l=P(v^k_l,u_{k+l})$. Using this and Proposition \ref{prop: Epsi_extension}(iv) we can write:
$$d_\psi(v^k_{l+1},v^k_l) = d_\psi(P(v^k_l,u_{k+l+1}),P(v^k_l,u_{k+l})) \leq d_\psi(u_{k+l+1}, u_{k+l})\leq \frac{1}{2^{k+l}}.$$

From Lemma \ref{lem: lemma mononton_seq} it follows now that each sequence $\{ v^k_l\}_{l \in \Bbb N}$  is $d_\psi$--convergening to some $v^k \in \mathcal E_\psi$ . By the same trick as above, we can write:
\begin{flalign*}
d_\psi(v^k,v^{k+1}) &=\lim_{l \to \infty}d_\psi(v^k_{l+1},v^{k+1}_l)= \lim_{l \to \infty}d_\psi(P(u_k,v^{k+1}_{l}),P(u_{k+1},v^{k+1}_l))\leq d_\psi (u_k,u_{k+1}) \leq \frac{1}{2^k},
\end{flalign*}\vspace{-0.7cm}
\begin{flalign*}
d_\psi(v^k,u_k) &=\lim_{l \to \infty}d_\psi(v^k_l,u_k)=\lim_{l \to \infty}d_\psi(P(u_k,v^{k+1}_{l-1}),P(u_k,u_k))\leq\lim_{l \to \infty}d_\psi(v^{k+1}_{l-1},u_k)\\
&=\lim_{l \to \infty}d_\psi(P(u_{k+1},v^{k+2}_{l-2}),u_k)\leq \lim_{l \to \infty}d_\psi(P(u_{k+1},v^{k+2}_{l-2}),u_{k+1}) + d_\psi(u_{k+1},u_k)\\
&\leq \lim_{l \to \infty} \sum_{j=k}^{l+k} d_\psi (u_j,u_{j+1}) \leq \frac{1}{2^{k-1}}.
\end{flalign*}
Consequently, $\{v^k\}_{k \in \Bbb N}$ is an increasing $d_\psi$--bounded $d_\psi$--Cauchy sequence that is equivalent to $\{u_k\}_{k \in \Bbb N}$. By Lemma \ref{lem: lemma mononton_seq} there exists $v \in \mathcal E_\psi$ such that $d_\psi(v_k,v) \to 0$, which in turn implies that $d_\psi(u_k,v)\to 0$, finishing the proof.
\end{proof}

\section{An analytic expression governing the $d_\psi$ metric}

As another application of the Pythagorean formula we will show that the $d_\psi$ metric is comparable to a concrete analytic expression, reminiscent of the analogous result for high energy classes \cite[Theorem 3]{Da15}:

\begin{theorem} \label{thm: Energy_Metric_Eqv} For any $u_0,u_1 \in \mathcal E_\psi$ we have
\begin{equation}\label{eq: Energy_Metric_Eqv}
d_\psi(u_0,u_1) \leq \int_X \psi(u_0 - u_1) \o_{u_0}^n + \int_X \psi(u_0 - u_1) \o_{u_1}^n\leq {2^{2n + 5}} d_\psi(u_0,u_1).
\end{equation}
\end{theorem}
\begin{proof} To obtain the first estimate we use the triangle inequality and Proposition \ref{prop: Epsi_extension}(iii):
\begin{flalign*}
d_\psi(u_0,u_1) &\leq
d_\psi(u_0,\max(u_0,u_1)) + d_\psi(\max(u_0,u_1),u_1)\\
&\leq \int_X \psi(u_0-\max(u_0,u_1)) \o_{u_0}^n + \int_X \psi(\max(u_0,u_1)-u_1) \o_{u_1}^n\\
&= \int_{\{u_1 > u_0\}} \psi(u_0-u_1) \o_{u_0}^n + \int_{\{u_0>u_1\}} \psi(u_0-u_1) \o_{u_1}^n\\
&\leq  \int_X \psi(u_0-u_1) \o_{u_0}^n + \int_X \psi(u_0-u_1) \o_{u_1}^n.
\end{flalign*}
Now we deal with the second estimate in \eqref{eq: Energy_Metric_Eqv}. By the next result, Proposition \ref{prop: Epsi_extension}(ii) and Proposition \ref{prop: Epsi_extension}(iii) we can write
\begin{flalign*}
2^{n+2}d_\psi(u_0,u_1) &\geq d_\psi\Big(u_0,\frac{u_0 + u_1}{2}\Big) \geq d_\psi\Big(u_0,P\Big(u_0,\frac{u_0 + u_1}{2}\Big)\Big)\\
&\geq \int_X \psi \Big(u_0 - P\Big(u_0,\frac{u_0 + u_1}{2}\Big)\Big) \o_{u_0}^n.
\end{flalign*}
By a similar reasoning as above, and the fact that $2^n \o^n_{(u_0 + u_1)/2} \geq \o^n_{u_0}$ we can write:
\begin{flalign*}
2^{n+2}d_\psi(u_0,u_1) &\geq d_\psi\Big(u_0,\frac{u_0 + u_1}{2}\Big) \geq d_\psi\Big(\frac{u_0+u_1}{2},P\Big(u_0,\frac{u_0 + u_1}{2}\Big)\Big)\\
&\geq \int_X \psi\Big(\frac{u_0+u_1}{2} - P\Big(u_0,\frac{u_0 + u_1}{2}\Big)\Big) \o_{(u_0 + u_1)/2}^n\\
&\geq \frac{1}{2^n} \int_X \psi \Big(\frac{u_0+u_1}{2} - P\Big(u_0,\frac{u_0 + u_1}{2}\Big)\Big) \o^n_{u_0}.
\end{flalign*}
Adding the last two estimates, and using sublinearity and concavity of $\psi$ (Lemma \ref{lem: sublinear}) we obtain:
\begin{flalign*}
2^{2n+3} d_\psi(u_0,u_1)&\geq  \int_X \psi\Big(u_0 - P\Big(u_0,\frac{u_0 + u_1}{2}\Big)\Big)+\psi \Big(P\Big(u_0,\frac{u_0 + u_1}{2}\Big)-\frac{u_0+u_1}{2}\Big) \Big) \o^n_{u_0}\\
&\geq \int_X \psi\Big(\frac{u_0 - u_1}{2}\Big)\o^n_{u_0} \geq \frac{1}{2}\int_X \psi(u_0 - u_1)\o^n_{u_0}.
\end{flalign*}
By symmetry we also have $2^{2n +4}d_\psi(u_0,u_1) \geq \int _X \psi(u_0 - u_1)\o_{u_1}^n$, and adding these last two estimates together the second inequality in \eqref{eq: Energy_Metric_Eqv}  follows.
\end{proof}

\begin{lemma} \label{lem: halwayest} Suppose $u_0,u_1 \in \mathcal E_\psi$. Then we have
$$d_\psi\Big(u_0,\frac{u_0+u_1}{2}\Big) \leq 2^{n+2}d_\psi(u_0,u_1).$$
\end{lemma}

\begin{proof} Using Proposition \ref{prop: Epsi_extension}(ii) and (iii) we can start writing:
\begin{flalign*}
d_\psi & \Big(u_0,\frac{u_0 + u_1}{2}\Big) = d_\psi\Big(u_0, P\Big(u_0,\frac{u_0 + u_1}{2}\Big)\Big) + d_\psi\Big(\frac{u_0 + u_1}{2},P\Big(u_0,\frac{u_0 + u_1}{2}\Big)\Big)\\
&\leq d_\psi(u_0, P(u_0,u_1)) + d_\psi\Big(\frac{u_0 + u_1}{2},P(u_0,u_1)\Big)\\
&\leq  \int_X \psi(u_0 - P(u_0,u_1)) \o^n_{P(u_0,u_1)} + \int_X\psi\Big(\frac{u_0+u_1}{2} - P(u_0,u_1)\Big)\o^n_{P(u_0,u_1)}\\
&\leq \int_X \psi(u_0 - P(u_0,u_1)) \o^n_{P(u_0,u_1)} + \int_X\psi(\max(u_1,u_0) - P(u_0,u_1))\o^n_{P(u_0,u_1)}\\
&\leq \int_X \Big((1+ \mathbbm{1}_{\{u_0 \geq u_1\}})\psi(u_0 - P(u_0,u_1)) +\mathbbm{1}_{\{u_1 \geq u_0\}}\psi(u_1 - P(u_0,u_1))\Big)\o^n_{P(u_0,u_1)}\\
& \leq 2^{n+2} \big(d_\psi(u_0, P(u_0,u_1)) + d_\psi(u_1,P(u_0,u_1)) \big)= 2^{n+2}d_\psi(u_0,u_1),
\end{flalign*}
where in the second line we have used the first claim of Lemma \ref{lem: mon_triplet_ineq}  and the fact that $P(u_0,u_1) \leq P(u_0,(u_0 + u_1)/2)$, in the third and sixth line Proposition \ref{prop: Epsi_extension}(iii), and in the last equality we have used Proposition \ref{prop: Epsi_extension}(ii).
\end{proof}

\paragraph{Declarations.} We declare no conflict of interest with any organization.

\addtocontents{toc}{\protect\enlargethispage{2\baselineskip}}
\addcontentsline{toc}{section}{References}

\footnotesize
\let\OLDthebibliography\thebibliography
\renewcommand\thebibliography[1]{
  \OLDthebibliography{#1}
  \setlength{\parskip}{1pt}
  \setlength{\itemsep}{1pt}
}

\vspace{0.3in}
\noindent {\sc Department of Mathematics, University of Maryland}\vspace{0.1cm}\\
{\tt tdarvas@umd.edu}\vspace{0.1in}
\end{document}